\documentclass[final]{siamltex}
\usepackage{cite}
\usepackage{amssymb,latexsym,amsmath}
\usepackage{mathrsfs}
\usepackage{enumitem}
\usepackage{caption}
\usepackage{inputenc}
\allowdisplaybreaks

\usepackage{graphics} 
\usepackage{epsfig} 
\usepackage{verbatim}

\newtheorem{example}{Example}[section]
\newtheorem{remark}{Remark}[section]

\usepackage{color}

\usepackage{amsfonts}
\usepackage{float}
\usepackage{multirow}
\usepackage{tikz}
\usetikzlibrary{shapes.geometric, arrows}
\allowdisplaybreaks

\tikzstyle{terminal}=[rectangle, rounded corners, minimum width=1cm, minimum height=1cm,text centered, draw=black]

\tikzstyle{terminal_2}=[rectangle, rounded corners, minimum width=0.5cm, minimum height=0.5cm,text centered, draw=black]

\tikzstyle{terminal_3}=[rectangle, rounded corners, minimum width=0.5cm, minimum height=0.5cm,text centered, draw=red]

\tikzstyle{arrow} = [thick,->,>=stealth]

\tikzstyle{addblock} = [draw,circle]

\tikzstyle{cost_block}=[rectangle, minimum width=1cm, minimum height=1cm,text centered, draw=red]

\renewcommand{\baselinestretch}{1}

\title{Exponential Filter Stability via Dobrushin's Coefficient\thanks{Supported by the Natural Sciences and Engineering Research Council of Canada.}}

\author{Curtis McDonald\thanks{Dept. of Stat and Data Sciecne, Yale University, United States of America. \newline Email: curtis.mcdonald@yale.edu} 
  \and 
  Serdar Y\"uksel\thanks{Dept. of Math and Stats, Queen's University, Canada. \newline Email: yuksel@queensu.ca}}

\begin{document}
\maketitle
\begin{abstract}
Filter stability is a classical problem in the study of partially observed Markov processes (POMP), also known as hidden Markov models (HMM). For a POMP, an incorrectly initialized non-linear filter is said to be (asymptotically) stable if the filter eventually corrects itself as more measurements are collected. Filter stability results in the literature that provide rates of convergence typically rely on very restrictive mixing conditions on the transition kernel and measurement kernel pair, and do not consider their effects independently. In this paper, we introduce an alternative approach using the Dobrushin coefficients associated with both the transition kernel as well as the measurement channel. Such a joint study, which seems to have been unexplored, leads to a concise analysis that can be applied to more general system models under relaxed conditions: in particular, we show that if $(1 - \delta(T))(2-\delta(Q)) < 1$, where $\delta(T)$ and $\delta(Q)$ are the Dobrushin coefficients for the transition and the measurement kernels, then the filter is exponentially stable. Our findings are also applicable for controlled models.
\end{abstract}


\section{Introduction}

In the study of partially observed Markov processes (POMP), also known as hidden Markov models (HMM), we have a hidden state process that is developing over time and an observer who sees noisy measurements of the state. The observer computes conditional estimates of the state given their measurements to date sequentially through a non-linear filtering equation. The filter is computed in a recursive fashion using a Bayesian update, however this recursion is dependent on the observer's prior (with respect to the unobserved initial state) before he/she has made any measurements. If the observer has the wrong prior, the filter they compute will not match the true filter and we say the filter has been incorrectly initialized. Filter stability is concerned with the merging of the true filter and the incorrectly initialized filter as the observer collects more measurements. That is, even if the observer has the wrong prior for the system, with enough measurements this mistake will be corrected asymptotically. 

Asymptotic stability, where the filters merge as time goes on but at no specified rate, may be problematic since one cannot guarantee sufficient merging for a fixed finite time. For many applications, it is desirable to attach a rate of merging for filter stability, so that in finite time one can guarantee how ``close'' the false filter is to the true filter. As we will note in the literature review, there are such stability results in the literature however they rely on rather restrictive mixing conditions on the transition kernel. 

In this paper, we propose a new sufficient condition for exponential stability using Dobrushin coefficients associated with both the transition kernel as well as the measurement channel. Such a joint study seems to have been unexplored and leads to concise explicit conditions on filter stability which can be applied to general system models under more relaxed conditions.

\subsection{Notation and Preliminaries}\label{sec:note}
In the following, we will discuss the control-free model setup. The controlled case will be considered in Section \ref{controlledCase}.

Let $\mathcal{X},\mathcal{Y}$ be Polish (that is, complete, separable, metric) spaces equipped with their Borel sigma fields $\mathcal{B}(\mathcal{X})$ and $\mathcal{B}(\mathcal{Y})$. $\mathcal{X}$ will be called the state space, and $\mathcal{Y}$ the measurement space. 

Given a measurable space $(\mathcal{X},\mathcal{B}(\mathcal{X}))$ we denote the space of probability measures on this space as $\mathcal{P}(\mathcal{X})$. We will denote random variables by capital letters and their realizations with lower case letters. Further, we will express contiguous sets of random variables such as $Y_{0},Y_{1},\cdots, Y_{n}$ with a subscript $Y_{[0,n]}$ indicating the starting and ending index of the collection. Infinite sequences $Y_{0},Y_{1},\cdots$ will be expressed as $Y_{[0,\infty)}$.
 We then define two probability kernels, the transition kernel $T$ and the measurement kernel $Q$:
\begin{align*}
&T:\mathcal{X} \to \mathcal{P}(\mathcal{X})&Q:\mathcal{X} \to \mathcal{P}(\mathcal{Y})\\
&~~x \mapsto T(dx'|x) &~~x \mapsto Q(dy|x)
\end{align*}
where for a set $A \in \mathcal{B}(\mathcal{Y})$ we write $Q(x,A)=\int_{A}Q(dy|x)$.  For these kernel operators, we can overload the notation to define them as mappings from a space of probability measures to another space of probability measures as follows
\begin{align*}
&T:\mathcal{P}(\mathcal{X}) \to \mathcal{P}(\mathcal{X})
&Q:\mathcal{P}(\mathcal{X}) \to \mathcal{P}(\mathcal{Y})~~~~~~~~~~~~~~~~\\
&~~~~\pi(dx) \mapsto \int_{\mathcal{X}}T(dx'|x)\pi(dx)
&\pi(dx) \mapsto \int_{\mathcal{Y}}Q(dy|x)\pi(dx)
\end{align*}
In practice, the form of the kernel operator is clear via context if the input is a probability measure or an element of the state space. Note that $T$ and $Q$ are time invariant kernels in a POMP as we study.

A POMP is initialized with a state $x_{0} \in \mathcal{X}$ drawn from a prior measure $\mu$ on $(\mathcal{X},\mathcal{B}(\mathcal{X}))$. However, the state is not available at the observer, instead the observer sees the sequence $Y_{n}\sim Q(dy|X_{n})$. That is, each $Y_{n}$ is a noisy measurement of the hidden random variable $X_{n}$ via the measurement channel $Q$. We then have for any set $A \in \mathcal{B}(\mathcal{X} \times \mathcal{Y})$,
\begin{align}
P\bigg((X_{0},Y_{0}) \in A\bigg)=\int_{A}Q(dy|x)\mu(dx)
\end{align}
and the POMP updates via the transition kernel  $T:\mathcal{X} \to \mathcal{P}(\mathcal{X})$
\begin{align}
P((X_{n},Y_{n}) \in A|(X,Y)_{[0,n-1]}=(x,y)_{[0,n-1]}) 
=\int_{A}Q(dy|x_{n})T(dx_{n}|x_{n-1})
\end{align}
It follows that $\{(X_{n},Y_{n})\}_{n=0}^{\infty}$ itself is a Markov chain, and we will denote $P^{\mu}$ as the probability measure on $\Omega=\mathcal{X}^{\mathbb{Z}_{+}}\times \mathcal{Y}^{\mathbb{Z}_{+}}$, endowed with the product topology where $X_{0} \sim \mu$ (this of course means $\omega \in \Omega$ is a sequence of states and measurements $\omega=\{(x_{i},y_{i})\}_{i=0}^{\infty}$). A diagram of the flow of the POMP is seen in Figure \ref{fig:POMP}. The nodes represent random variables, and the arrows are labelled with the kernel that defines the conditional measure between two random variables. That is, the distribution of $Y_{1}$, conditioned on the past events, is fully determined by the realization of $X_{1}$ and the measurement channel $Q$, and the distribution of $X_{2}$, conditioned on the past events, is fully determined by the realization of $X_{1}$ and the transition kernel $T$.

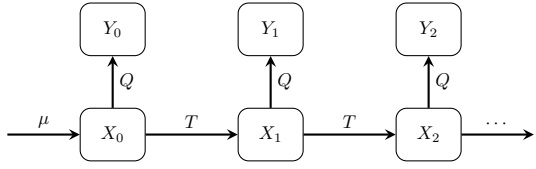
\begin{figure}
\begin{center}
\begin{tikzpicture}[scale=0.7, every node/.style={transform shape}]


\node (X0) [terminal,  text width = 1cm] {$X_{0}$};
\node (X1) [terminal, right of = X0, xshift = 2cm, text width = 1cm] {$X_{1}$};
\node (X2) [terminal, right of = X1, xshift = 2cm, text width = 1cm] {$X_{2}$};


\node (Y0) [terminal, above of = X0, yshift = 1cm, text width = 1cm] {$Y_{0}$};
\node (Y1) [terminal, above of = X1, yshift = 1cm, text width = 1cm] {$Y_{1}$};
\node (Y2) [terminal, above of = X2, yshift = 1cm, text width = 1cm] {$Y_{2}$};

\draw[arrow] (X0) -- node[anchor=south] {$T$} (X1);
\draw[arrow] (X1) -- node[anchor=south] {$T$} (X2);

\draw[arrow] (X0) -- node[anchor=west] {$Q$} (Y0);
\draw[arrow] (X1) -- node[anchor=west] {$Q$} (Y1);
\draw[arrow] (X2) -- node[anchor=west] {$Q$} (Y2);

\draw[arrow] (-2 cm, 0 cm) -- node[anchor=south] {$\mu$} (X0);

\draw[arrow] (X2) -- node[anchor=south] {$\cdots$} (8 cm, 0);
\end{tikzpicture}
\caption{Chain of Implications in POMP}
\label{fig:POMP}
\end{center}
\end{figure}

\begin{definition} We define the filter as the sequence of conditional probability measures
\begin{align}
\pi_{n}^{\mu}(\cdot)=P^{\mu}(X_{n} \in \cdot|Y_{[0,n]}) ~~n\in \{0,1,2,\cdots\}
\end{align}
\end{definition}
Calculating the filter realizations can be performed in a recursive manner. That is, given the previous filter realization $\pi_{n}^{\mu}\in \mathcal{P}(\mathcal{X})$ and a new observation $y_{n+1}\in \mathcal{Y}$ we can compute the next filter realization $\pi_{n+1}^{\mu}$ via the filter update function $\phi:\mathcal{P}(\mathcal{X}) \times \mathcal{Y} \to \mathcal{P}(\mathcal{X})$. 

Often one assumes that there exists a dominating measure $\lambda \in \mathcal{P}(\mathcal{Y})$ and for every $x \in \mathcal{X}$, $Q(dy|x) \ll \lambda$. Note that ``$\ll$'' means absolute continuity, so that for any set $A \in \mathcal{B}(\mathcal{Y})$ we have $\lambda(A)=0 \implies Q(x,A)=0~\forall x \in \mathcal{A}$. 
Then we say $Q$ is dominated and there exists a Radon Nikodym derivative for $Q(dy|x)$ with respect to $\lambda(dy)$ for each $x$, which is the conditional probability density function (pdf) or likelihood function $\frac{d Q}{d\lambda}(x,y)=g(x,y)$. Then we can define the \textit{Bayesian update operator}
\begin{align*}
\psi:& \mathcal{P}(\mathcal{X})\times \mathcal{Y} \to \mathcal{P}(\mathcal{X})\cup\{0\}\\
&(\pi(dx),y) \mapsto 
\begin{cases}
\frac{g(x,y)\pi(dx)}{\int_{\mathcal{X}}g(x,y)\pi(dx)}&\text{if}~\int_{\mathcal{X}}g(x,y)\pi(dx)>0\\
0&\text{else}
\end{cases}
\end{align*}

We will later call $N^{\mu}(y)=\int_{\mathcal{X}}g(x,y)\pi(dx)$ the normalizing constant. If $(X,Y) \sim P^{\mu}$ where $P^{\mu}((X,Y) \in (A \times B))=\int_{B}\int_{A}Q(dy|x)\mu(dx)$ then $N^{\mu}(Y)$ is non-zero with $P^{\mu}$ probability 1. That is, the set $B=\{y|N^{\mu}(y)=0\}$ has zero probability since
\begin{align*}
P^{\mu}(Y \in B)=\int_{B}P^{\mu}(dy)&=\int_{B}\int_{\mathcal{X}}g(x,y)\mu(dx)dy=\int_{B}N^{\mu}(y)dy=0
\end{align*}
additionally, for any other prior with $\mu \ll \nu$, we also have that $N^{\nu}(Y)$ is non-zero with $P^{\mu}$ probability 1. Thus inside of $P^{\mu}$ expectations we can consider $N^{\mu}(Y)$ and $N^{\nu}(Y)$ as being non-zero.

We can  then explicitly write the filter update operator as the composition of the Bayesian update operator with the transition kernel
\begin{align}
\pi_{n+1}^{\mu}(dx)&=\phi(\pi_{n}^{\mu},y_{n+1})(dx)=\psi(T(\pi_{n}^{\mu}),y_{n+1})(dx)
=\frac{g(x,y_{n+1})\int_{\mathcal{X}}T(dx|x')\pi_{n}^{\mu}(dx')}{\int_{\mathcal{X}}g(x,y_{n+1})\int_{\mathcal{X}}T(dx|x')\pi_{n}^{\mu}(dx')}\label{filter_update_eqn}
\end{align}
where (\ref{filter_update_eqn}) is often referred to as the filter update equation in the literature.

Since the filter update is a recursive process, it is sensitive to the initial distribution of $X_{0}$ which is the starting point of the recursion. Suppose that an observer computes the non-linear filter assuming that the initial prior is $\nu$, when in reality the prior distribution is $\mu$. The observer receives the measurements and computes the filter $\pi_{n}^{\nu}$ for each $n$, but the measurement process is generated according to the true measure $\mu$. The question we are interested in is that of filter stability, namely, if we have two different initial probability measures $\mu$ and $\nu$, when do we have that the filter processes $\pi_{n}^{\mu}$ and $\pi_{n}^{\nu}$ merge in some appropriate sense as $n \to \infty$?

\begin{definition}
For two probability measures $P,Q$ we define the total variation norm as
$\|P-Q\|_{TV}=\sup_{\|f\|_{\infty}\leq 1} \left|\int f dP-\int d dQ \right|$ where $f$ is assumed measurable and bounded with norm 1.
\end{definition}

\begin{definition}
A POMP is said to be exponentially stable in total variation in expectation if there exists a coefficient $0< \alpha< 1$ such that for any $\mu \ll \nu$ we have
\begin{align*}
E^{\mu}[\|\pi_{n+1}^{\mu}-\pi_{n+1}^{\nu}\|_{TV}]\leq \alpha E^{\mu}[\|\pi_{n}^{\mu}-\pi_{n}^{\nu}\|_{TV}]~~n \in \{0,1,\cdots\}
\end{align*}
\end{definition}

 Before we state our main result and supporting results, a brief literature review is presented next. Our main results are presented in Section \ref{sec_main}, with Theorem \ref{mixed_Dobrushin} providing a sufficient condition for exponential stability of the filter. In Section \ref{controlledCase}, we explain how these results can easily be applied to control models. A simple but useful application of the new approach is presented in Section \ref{CaseStudySection}, and concluding remarks in Section \ref{sec:conclusion}.

\section{Literature Review}

Filter stability is a very important subject, and consequently, one that has been studied extensively. We refer the reader to \cite{chigansky2009intrinsic, di2005ergodicity, le2004stability, chigansky2010complete,ocone1996asymptotic,douc2010forgetting, budhiraja1999exponential, crisan2008stability} for a comprehensive review and a collection of different approaches. As discussed in \cite{chigansky2009intrinsic}, filter stability arises via two separate mechanisms:
\begin{enumerate}
\item The transition kernel is in some sense {\it sufficiently} ergodic, forgetting the initial measure and therefore passing this insensitivity (to incorrect initializations) on to the filter process.
\item The measurement channel provides sufficient information about the underlying state, allowing the filter to  track the true state process.
\end{enumerate}
For a review of the methods utilizing the second mechanism above involving observability related aspects, we refer the refer to the very detailed literature reviews in \cite{chigansky2009intrinsic} and \cite{mcdonald2018stability}. 
 
Most of the literature has focused on the first of the two mechanisms noted above by showing that the transition kernel $T$ is \textit{sufficiently ergodic} \cite{chigansky2009intrinsic}, forgetting the initial measure as time goes on. By ergodicity, here we mean that the successive applications of the transition kernel $T$ brings any two different priors closer together through the filter update equation with increasing time. To achieve this end, results in the literature \cite{atar1997exponential,del2001stability, le2004stability} and various relaxations as in \cite{chigansky2004stability} or \cite[Theorems 2.1 and 2.2]{chigansky2009intrinsic} utilize some form of mixing, pseudo-mixing, or a similar condition on the transition kernel. A general mixing condition is along the lines of the following:

\begin{definition}\cite[Definition 3.2]{le2004stability}\label{mixingDefinition}
A kernel $K:\mathcal{S}_{1} \to \mathcal{P}(\mathcal{S}_{2})$ is called mixing if there exists a finite non-negative measure $\lambda \in \mathcal{P}(\mathcal{S}_{2})$ and $0< \epsilon\leq 1$ such that $\forall~A \in \mathcal{B}(\mathcal{S}_{2}), s \in \mathcal{S}_{1}$ $
\epsilon \lambda(A)\leq K(s,A)\leq \frac{1}{\epsilon}\lambda(A)$
\end{definition}
Such a mixing condition is a very strong assumption on a kernel. For example, a kernel on a finite probability space (which is a stochastic matrix) is mixing if and only if each column of the matrix is fully zero or fully non-zero. For example the following matrix is \textit{not} a mixing kernel.
\begin{footnotesize}
\begin{align*}
\begin{pmatrix}
0&0.25&0.75\\
0.25&0.25&0.5\\
0&0.1&0.9
\end{pmatrix}
\end{align*}
\end{footnotesize}

For a kernel $K:\mathcal{S}_{1} \to \mathcal{P}(\mathcal{S}_{2})$ with dominating measure $\lambda$ and likelihood function $k(s_{2}|s_{1})$, the kernel is mixing if and only if there exists two enveloping functions $f_{1},f_{2}\in L^{1}(\lambda)$ such that
\begin{align*}
0< a\leq \frac{f_{1}(s_{2})}{f_{2}(s_{2})}\leq b <\infty~\forall s_{2}\in \mathcal{S}_{2}\\
f_{1}(s_{2})\leq k(s_{2}|s_{1})\leq f_{2}(s_{2})~\forall s_{1}\in \mathcal{S}_{1},s_{2}\in \mathcal{S}_{2}
\end{align*}
For example, if $K:\mathbb{R} \to \mathcal{P}(\mathbb{R})$ where $K(dx'|x)\sim N(f(x),\sigma)$ where $\|f\|_{\infty}<\infty$ then $K$ is \textit{not} a mixing kernel.


The approach taken in \cite{le2004stability} utilizes the Hilbert metric to achieve stability.

\begin{definition}\cite[Definition 3.1]{le2004stability}
Two non-negative measures $\mu$ and $\nu$ on a measurable space $(S,\mathcal{F})$ are called comparable if $\exists~0< a \leq b$ such that $\forall A \in \mathcal{F}$,~
$
a\mu(A)\leq \nu(A)\leq b\mu(A)$.
\end{definition} 

\begin{definition}\cite[Definition 3.3]{le2004stability}
Let $\mu,\nu$ be two non-negative finite measures. We define the Hilbert metric on such measures as
\begin{align*}
h(\mu,\nu)&=\begin{cases}
\log\left(\frac{\sup_{A|\nu(A)>0}\frac{\mu(A)}{\nu(A)}}{\inf_{A|\nu(A)>0}\frac{\mu(A)}{\nu(A)}} \right)&\text{if}~\mu,\nu~\text{are comparable}\\
0&\text{if}~\mu=\nu=0\\
\infty&\text{else}
\end{cases}
\end{align*}
\end{definition}
We see that the Hilbert metric is only meaningful when $\mu$ and $\nu$ are comparable. Yet comparability implies mutual absolute continuity (i.e. $\mu \ll \nu$ and $\nu \ll \mu$) and therefore that the Radon Nikodym derivatives $\frac{d\mu}{d\nu}$ and $\frac{d\nu}{d\mu}$ exist, and furthermore, that these derivatives are bounded from above and below away from zero. In this case, we have that
$
h(\mu,\nu)=\log \left(\left\|\frac{d\mu}{d\nu} \right\|_{\infty}\left\|\frac{d\nu}{d\mu} \right\|_{\infty} \right)
$
when the measures are comparable.
The Hilbert metric is a projective distance, meaning if we scale either of the measures by a constant it will not change the Hilbert metric. This makes the metric very useful when studying the Bayesian update operator $\psi$ since the denominator in a Bayesian update is a non-linear scaling operator, while the numerator is a linear operator.
\begin{theorem}\cite[Corollary 4.2]{le2004stability}
Assume the measurement channel is dominated and has a likelihood function. Let $\bar{\phi}$ represent the un-normalized filter update,
$
\bar{\phi}(\mu,y)(dx)=g(x,y)T(\mu)(dx)
$,
which is a kernel mapping to the space of non-negative finite measures and not necessarily the space of probability measures. If $\bar{\phi}$ is a mixing Kernel with coefficient $\epsilon>0~\forall y \in \mathcal{Y}$ then
\begin{align}
\|\pi^{\mu}_{n+m}-\pi_{n+m}^{\nu}\|_{TV}\leq \left(\frac{2}{\log(3)\epsilon^{2}} \right)\left(\frac{1-\epsilon^{2}}{1+\epsilon^{2}} \right)^{m-1}\|\pi_{n}^{\mu}-\pi_{n}^{\nu}\|_{TV}\label{Hilbert_contract}
\end{align}
\end{theorem}
Note that if $T$ is a mixing kernel with coefficient $\epsilon$, then $\bar{\phi}$ is as well but this can also be achieved without $T$ begin mixing, see \cite[Example 3.10]{le2004stability}. However, requiring $\bar{\phi}$ to be a mixing kernel is a very restrictive assumption. Often, such a condition is not applicable for applications with a non-compact state space. We can also note that the exponential coefficient $\frac{1-\epsilon^{2}}{1+\epsilon^{2}}$ may be close to $1$ for many reasonable values of $\epsilon \ll 1$ and hence may lead to a very slow rate of decay.


In short, most exponential stability results in the literature rely on the mixing condition which may be prohibitive for many applications, as noted in \cite[Section 4.3.6]{cappe2005inference} this is not a desirable approach to filter stability. We would like to find an approach that does not rely on this condition. 

Instead of such a strong mixing condition, we will introduce a new approach based {\bf on a joint contraction property of the Bayesian filter update and measurement update steps through the Dobrushin coefficient}: The only references, to our knowledge, where the Dobrushin coefficient is utilized are \cite{del2001stability} and \cite[Section 4.3]{cappe2005inference}, however a careful look at these contributions ultimate rely on mixing conditions \cite[Assumption 4.3.21,4.3.24]{cappe2005inference}, and these do not consider the effect of the measurement channel to refine the bounds. Our approach leads to a concise derivation through a direct approach of the Dobrushin coefficients and leads to more relaxed characterizations as we take into account the measurement updates as well.

\section{Main Result}\label{sec_main}

Our approach is to study when the filter update operator $\phi$ is a contraction in expectation, that is
\begin{align*}
E^{\mu}[\|\phi(\pi_{n}^{\mu},y_{n+1})-\phi(\pi_{n}^{\nu},y_{n+1})\|_{TV}]\leq \alpha \|\pi_{n}^{\mu}-\pi_{n}^{\nu}\|_{TV}
\end{align*}
for some $\alpha<1$. We will go about this by studying the Dobrushin coefficients of $T$ and $Q$.

\begin{definition}\cite[Equation 1.16]{dobrushin1956central}
For a kernel operator $K:S_{1} \to \mathcal{P}(S_{2})$ we define the Dobrushin coefficient as:
\begin{align}
\delta(K)&=\inf\sum_{i=1}^{n}\min(K(x,A_{i}),K(y,A_{i}))\label{Dob_def}
\end{align}
where the infimum is over all $x,y \in S_{1}$ and all partitions $\{A_{i}\}_{i=1}^{n}$ of $S_{2}$.
\end{definition}
 Note this definition holds for continuous or finite/countable spaces $S_{1}$ and $S_{2}$ and $0\leq \delta(K)\leq 1$ for any kernel operator.
The Dobrushin coefficient is conceptually a measure on how similar or different the different conditional measures $K(ds_{2}|s_{1}),K(ds_{2}|s_{1}')$ are for different $s_{1},s_{1}'$ (different conditionals). If the measures are similar, the coefficient is close to $1$ and if they are different, it is close to 0. Let us look at two examples
\begin{example}[Finite Space Setup]
Assume $S_{1}$ and $S_{2}$ are finite spaces, then $K$ is a $|S_{1}|$ by $|S_{2}|$ stochastic matrix. The Dobrushin coefficient is the minimum over any two rows where we sum the minimum elements among those rows. If we have the matrix
\begin{align*}
K=\begin{pmatrix}
0&\frac{1}{3}&\frac{2}{3}\\
\frac{1}{2}&\frac{1}{2}&0\\
\frac{1}{3}&\frac{1}{3}&\frac{1}{3}
\end{pmatrix}
\end{align*}
If we pick the first and second row, the sum of the minimum elements is $\frac{1}{3}$. If we pick the first and third rows, it is $\frac{2}{3}$ and the second and third rows it is $\frac{2}{3}$. Therefore the Dobrushin coefficient is $\frac{1}{3}$.
\end{example}

\begin{example}[Continuous Space Setup]
Assume for simplicity $S_{1}=S_{2}=\mathbb{R}$ and the dominating measure is the Lebesgue measure. Then we have a conditional pdf $k(s_{2}|s_{1})$. For any choice of $s_{1}$ and $s_{1}'$, the minimizing partition is two sets: one set where $k(s_{2}|s_{1})>k(s_{2}|s_{1}')$ and it's compliment. The result is then the area under the overlap of the two pdf's, and the Dobrushin coefficient is the minimum of this overlap area for any two pdf's. A demonstration for two pdf's is provided in Figure \ref{fig:dob_example}, the overlap area is shaded in gray.

\begin{figure}
\begin{center}
\includegraphics[scale=0.8]{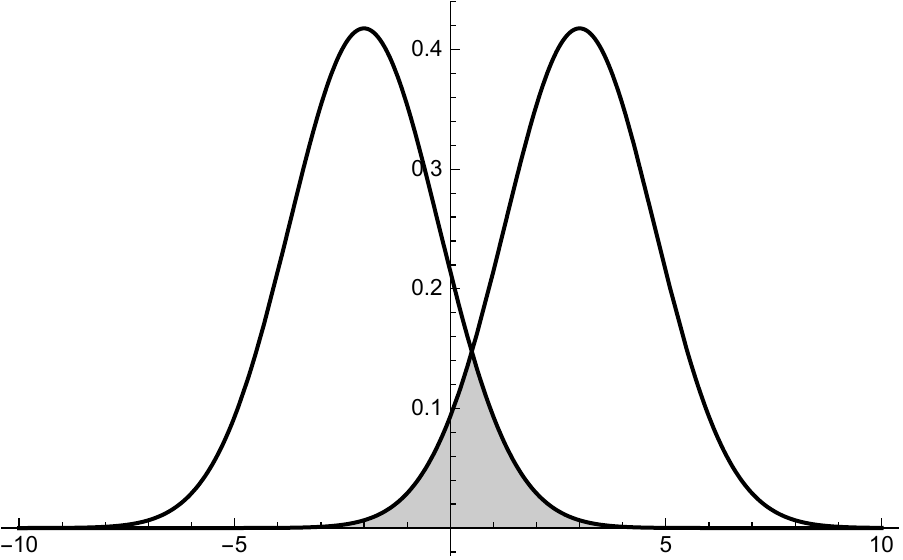}\label{fig:dob_example}
\caption{Example of Dobrushin coefficient calculation for dominated systems.}
\end{center}
\end{figure}

\end{example}
 
  The Dobrushin coefficient provides a contraction coefficient for kernel operators in total variation. For two probability measures $\pi,\pi' \in \mathcal{P}(S_{1})$  {\cite{dobrushin1956central}:
\begin{align*}
\|K(\pi)-K(\pi')\|_{TV} \leq (1-\delta(K))\|\pi-\pi'\|_{TV}
\end{align*}
As was discussed in Section \ref{sec:note}, the filter update operator $\phi$ is a composition of the transition kernel $T$ and the Bayesian update operator $\psi$. The transition operator $T$ is a contraction mapping with coefficient $(1-\delta(T))$, which potentially could be $1$. Assume that it is less than 1, then without the Bayes update the transition operator would bring measures together with each successive application. However, the Bayes operator is in general not a contraction, and can in fact increase the expected total variation distance between posteriors compared to the priors. 
\begin{example}\label{bayes_expand}
Consider as a simple example the priors and measurement kernel 
\begin{align*}
\mu=(0.05,0.65,0.3)
\quad \nu=(0.2, 0.65, 0.15)
\quad Q=\begin{pmatrix}
0.1&0.3&0.6\\
0.5&0.3&0.2\\
0.9&0.1&0
\end{pmatrix}
\end{align*}
the original total variation $\|\mu-\nu\|_{TV}$ distance is 0.3, but the expected distance of the posteriors is $0.3728$.
\end{example}

 We are therefore not guaranteed that the composition of the two operators $T$ and $\psi$ is a contraction. However, if we have an upper bound on
\begin{align*}
\frac{\int_{\mathcal{X}}\int_{\mathcal{Y}}\|\psi(\mu,y)-\psi(\nu,y)\|_{TV}Q(dy|x)\mu(dx)}{\|\mu-\nu\|_{TV}}
\end{align*}
then if $\delta(T)$ is sufficiently large, the possible expansion property of $\psi$ is dominated by the contraction property of $T$ and the composed operator $\phi$ is itself a contraction in expectation.

\begin{lemma}\label{Bayes_Dobrushin}
Consider a true prior $\mu$ and a false prior $\nu$ with $\mu \ll \nu$. Assume that the measurement channel $Q$ is dominated, then we have that
\begin{align*}
\int_{\mathcal{X}}\int_{\mathcal{Y}}\|\psi(\mu,y)-\psi(\nu,y)\|_{TV}]Q(dy|x)\mu(dx)\leq (2-\delta(Q))\|\mu-\nu\|_{TV}
\end{align*}
\end{lemma}
\begin{proof}
We will take a closer look at the operator $\psi$.  
For a general probability measure $\pi$ define the normalizing constant $
N^{\pi}(y)=\int_{\mathcal{X}}g(x,y)\pi(dx)
$.
As discussed in the notation section, $N^{\mu}(Y)$ and $N^{\nu}(Y)$ are non-zero with $P^{\mu}$ probability 1, and thus we will simply consider them as non-zero for the remainder of this proof.
\begin{small}
\begin{align*}
&\|\psi(\mu,y)-\psi(\nu,y)\|_{TV}=\sup_{\|f\|_{\infty} \leq 1}\left|\int_{\mathcal{X}}\frac{f(x)g(x,y)}{N^{\mu}(y)}\mu(dx)- \int_{\mathcal{X}}\frac{f(x)g(x,y)}{N^{\nu}(y)}\nu(dx)\right|\\
=&\sup_{\|f\|_{\infty} \leq 1}\left|\int_{\mathcal{X}}\frac{f(x)g(x,y)}{N^{\mu}(y)}\mu(dx)\pm \int_{\mathcal{X}}\frac{f(x)g(x,y)}{N^{\mu}(y)}\nu(dx)-\int_{\mathcal{X}}\frac{f(x)g(x,y)}{N^{\nu}(y)}\nu(dx) \right|\\
\leq& \sup_{\|f\|_{\infty}\leq 1}\frac{1}{N^{\mu}(y)}\left|\int_{X}f(x)g(x,y)(\mu-\nu)(dx)\right|+\sup_{\|f\|_{\infty}\leq 1}\left|\frac{1}{N^{\mu}(y)}-\frac{1}{N^{\nu}(y)} \right| \left|\int_{\mathcal{X}}f(x)g(x,y)\nu(dx) \right|\\
 \leq& \sup_{\|f\|_{\infty}\leq 1}\frac{1}{N^{\mu}(y)}\left|\int_{X}f(x)g(x,y)\left(\frac{d\mu}{d\nu}(x)-1 \right)\nu(dx)\right|+\left|\frac{N^{\nu}(y)-N^{\mu}(y)}{N^{\mu}(y)N^{\nu}(y)} \right| N^{\nu}(y)\\
\leq& \left(\frac{1}{N^{\mu}(y)} \right)\left(\left|N^{\mu}(y)-N^{\nu}(y) \right|+\int_{X}g(x,y)\left|1-\frac{d\mu}{d\nu}(x)\right|\nu(dx) \right)
\end{align*}
\end{small}
taking the expectation of this expression
\begin{small}
\begin{align*}
&\int_{X}\int_{Y}\|\psi(\mu,y)-\psi(\nu,y)\|_{TV}Q(dy|x)\mu(dx)=\int_{X}\int_{Y}\|\psi(\mu,y)-\psi(\nu,y)\|_{TV}g(x,y)\lambda(dy)\mu(dx)\\
=&\int_{Y}\|\psi(\mu,y)-\psi(\nu,y)\|_{TV}\left(\int_{X}g(x,y)\mu(dx)\right)\lambda(dy)\\
=&\int_{Y}\|\psi(\mu,y)-\psi(\nu,y)\|_{TV}N^{\mu}(y)\lambda(dy)\\
\leq& \int_{\mathcal{Y}}\left(\left|N^{\mu}(y)-N^{\nu}(y) \right|+\int_{X}g(x,y)\left|1-\frac{d\mu}{d\nu}(x)\right|\nu(dx)\right)\lambda(dy)\\
\leq& \int_{\mathcal{Y}}|N^{\mu}(y)-N^{\nu}(y)|\lambda(dy)+\int_{\mathcal{Y}}\int_{X}g(x,y)\left|1-\frac{d\mu}{d\nu}(x)\right|\nu(dx)\lambda(dy)\\
=&\int_{\mathcal{Y}}\left|\int_{\mathcal{X}}g(x,y)(\mu-\nu)(dx)\right|\lambda(dy)
+\int_{X}\left|1-\frac{d\mu}{d\nu}(x)\right|\left(\int_{\mathcal{Y}}g(x,y)\lambda(dy)\right)\nu(dx)
\end{align*}
\end{small}
Let us examine these two terms separately. For the second term, $g(x,y)$ is a probability density function for a fixed $x$, therefore it integrates to $1$ over $\lambda$ and we have
\begin{align*}
&\int_{X}\left|1-\frac{d\mu}{d\nu}(x)\right|\left(\int_{\mathcal{Y}}g(x,y)\lambda(dy)\right)\nu(dx)
=\int_{X}\left|1-\frac{d\mu}{d\nu}(x)\right|\nu(dx)=\|\mu-\nu\|_{TV}
\end{align*} 
for the first term, define the sets
\begin{align*}
S^{+}&=\{y|\int_{X}g(x,y)(\mu-\nu)(dx)>0\}\quad \quad
S^{-}=\{y|\int_{X}g(x,y)(\mu-\nu)(dx)\leq 0\}
\end{align*}
then we have
\begin{small}
\begin{align*}
&\int_{\mathcal{Y}}\left|\int_{\mathcal{X}}g(x,y)(\mu-\nu)(dx)\right|\lambda(dy)
=\int_{\mathcal{Y}}\left(1_{S^{+}}(y)-1_{S^{-}}(y) \right)\int_{\mathcal{X}}g(x,y)(\mu-\nu)(dx)\lambda(dy)
\end{align*}
\end{small}
We then have that $1_{S^{+}}(y)-1_{S^{-}}(y)$ is a measurable function of $y$ with infinity norm equal to 1, and in fact it achieves the supremum over all such functions. That is
\begin{align*}
&\int_{\mathcal{Y}}\left(1_{S^{+}}(y)-1_{S^{-}}(y) \right)\int_{X}g(x,y)(\mu-\nu)(dx)\lambda(dy)
=\sup_{\|f\|_{\infty}\leq 1}\left|\int_{\mathcal{Y}}f(y)\int_{\mathcal{X}}g(x,y)(\mu-\nu)(dx)\lambda(dy) \right|\\
=&\sup_{\|f\|_{\infty}\leq 1}\left|\int_{\mathcal{X}}\int_{\mathcal{Y}}f(y)g(x,y)\lambda(dy)(\mu-\nu)(dx) \right|
=\|Q(\mu)-Q(\nu)\|_{TV}\leq (1-\delta(Q))\|\mu-\nu\|_{TV}
\end{align*}
\end{proof}

Indeed, if we consider Example \ref{bayes_expand}, the Dobrushin coefficient of $Q$ is $0.2$, so our upper bound is 1.8 while the ratio provided is $\frac{0.3726}{0.3}=1.24$, less than our upper bound. More important though is pairing the Bayes update with a sufficiently contractive transition kernel.
\begin{theorem}\label{mixed_Dobrushin}
Assume that $\mu \ll \nu$ and that the measurement channel $Q$ is dominated. Then we have
\begin{align*}
E^{\mu}[\|\pi_{n+1}^{\mu}-\pi_{n+1}^{\nu}\|_{TV}]\leq (1-\delta(T))(2-\delta(Q))E^{\mu}[\|\pi_{n}^{\mu}-\pi_{n}^{\nu}\|_{TV}]
\end{align*}
\end{theorem}
\begin{proof}
\begin{align*}
&E^{\mu}[\|\pi_{n+1}^{\mu}-\pi_{n+1}^{\nu}\|_{TV}]
=E^{\mu}[\|\phi(\pi^{\mu}_{n},y_{n+1})-\phi(\pi_{n}^{\nu},y_{n+1})\|_{TV}]\\
&=E^{\mu}[\|\psi(T(\pi_{n}^{\mu}),y_{n+1})-\psi(T(\pi_{n}^{\nu}),y_{n+1})\|_{TV}]\\
&=\int_{\mathcal{Y}^{n+2}}\|\psi(T(\pi_{n}^{\mu}),y_{n+1})-\psi(T(\pi_{n}^{\nu}),y_{n+1})\|_{TV}P^{\mu}(dy_{[0,n+1]})\\
&=\int_{\mathcal{Y}^{n+1}}\int_{\mathcal{Y}}\|\psi(T(\pi_{n}^{\mu}),y_{n+1})-\psi(T(\pi_{n}^{\nu}),y_{n+1})\|_{TV}P^{\mu}(dy_{n+1}|y_{[0,n]})P^{\mu}(dy_{[0,n]})
\end{align*}
now we condition on $X_{n+1}$ as a hidden variable to break the conditioning into two parts.
\begin{small}
\begin{align*}
\int_{\mathcal{Y}^{n+1}}\int_{\mathcal{X}}\int_{\mathcal{Y}}\|\psi(T(\pi_{n}^{\mu}),y_{n+1})-\psi(T(\pi_{n}^{\nu}),y_{n+1})\|_{TV}P^{\mu}(dy_{n+1}|x_{n+1},y_{[0,n]})P^{\mu}(dx_{n+1}|y_{[0,n]})P^{\mu}(dy_{[0,n]})
\end{align*}
\end{small}
since $Y_{n+1}$ is fully determined by $X_{n+1}$ we have that $P^{\mu}(dy_{n+1}|x_{n+1},y_{[0,n]})=Q(dy_{n+1}|x_{n+1})$ and the measure $P^{\mu}(dx_{n+1}|y_{[0,n]})=T(\pi_{n}^{\mu})(dx_{n+1})$ is the filter put through the transition kernel. We then have
\begin{small}
\begin{align*}
&\int_{\mathcal{Y}^{n+1}}\left(\int_{\mathcal{X}}\int_{\mathcal{Y}}\|\psi(T(\pi_{n}^{\mu}),y_{n+1})-\psi(T(\pi_{n}^{\nu}),y_{n+1})\|_{TV}Q(dy_{n+1}|x_{n+1}) T(\pi_{n}^{\mu})(dx_{n+1})\right)P^{\mu}(dy_{[0,n]})\\
\end{align*}
\end{small}
Now consider the expression inside the brackets. $T(\pi_{n}^{\mu})$ is playing the role of a prior for $X_{n+1}$ before the observation $Y_{n+1}$ is made, and therefore this expression is exactly the form of an expected Bayes update as studied in Lemma \ref{Bayes_Dobrushin}. We can apply the Dobrushin bound on the Bayes update and we have:
\begin{small}
\begin{align*}
&\leq (2-\delta(Q))\int_{\mathcal{Y}^{n+1}}\|T(\pi_{n}^{\mu})-T(\pi_{n}^{\nu})\|_{TV}P^{\mu}(dy_{[0,n]})= (2-\delta(Q))(1-\delta(T))E^{\mu}[\|\pi_{n}^{\mu}-\pi_{n}^{\nu}\|_{TV}]
\end{align*}
\end{small}
\end{proof}

\begin{corollary}
Assume $\mu \ll \nu$ and that the measurement channel is $Q$ is dominated. If we have $
\alpha=(1-\delta(T))(2-\delta(Q)) < 1
$
then the filter is exponentially stable in total variation in expectation with coefficient $\alpha$ and
\begin{align*}
E^{\mu}[\|\pi_{n}^{\mu}-\pi_{n}^{\nu}\|_{TV}]\leq (2-\delta(Q))\left( \alpha^{n}\right)\|\mu-\nu\|_{TV}
\end{align*}
Furthermore, if $\delta(T)>\frac{1}{2}$ then $\alpha< 1$ and the POMP is exponentially stable regardless of the measurement kernel $Q$.
\end{corollary}
\begin{proof}
By recursive application of Theorem \ref{mixed_Dobrushin} we have
\begin{align*}
E^{\mu}[\|\pi_{n}^{\mu}-\pi_{n}^{\nu}\|_{TV}]\leq \alpha^{n}E^{\mu}[\|\pi_{0}^{\mu}-\pi_{0}^{\nu}\|_{TV}]
\end{align*}
$\pi_{0}^{\mu}$ is then the Bayesian update of $\mu$ under the first observation $Y_{0}$, therefore we apply Lemma \ref{Bayes_Dobrushin} and we have
\begin{align*}
\alpha^{n}E^{\mu}[\|\pi_{0}^{\mu}-\pi_{0}^{\nu}\|_{TV}]&=\alpha^{n}E^{\mu}[\|\psi(\mu,y_{0})-\psi(\nu,y_{0})\|_{TV}]\leq (2-\delta(Q))(\alpha^{n})\|\mu-\nu\|_{TV}
\end{align*}
Finally, recall that for any kernel $K$ we have $0\leq \delta(K)\leq 1$ therefore if we have $\delta(T)>\frac{1}{2}$
\begin{align*}
\alpha=(1-\delta(T))(2-\delta(Q))< \frac{1}{2}(2-\delta(Q))\leq \frac{2}{2}=1
\end{align*}
\end{proof}

\begin{remark}
In \cite[Equation 1.5]{chigansky2004stability} the authors provide a condition depending only on the transition kernel that results in exponential filter stability. This condition is a weakening of the typical mixing results in Definition 2.1, but still inherits similar issues about zero probability entries. For example, in a finite state space the condition \cite[Equation 1.5]{chigansky2004stability} requires that at least one row of the transition matrix is non-zero, while typical mixing requires all rows to be non-zero. In a continuous state space, the previously discussed example $K(dx'|x)\sim N(f(x),\sigma)$ where $\|f\|_{\infty}<\infty$ would also violate \cite[Equation 1.5]{chigansky2004stability}. The condition does not imply a non-zero Dobrushin coefficient nor it is implied by it, and thus gives a complementary sufficient condition for exponential stability. Our condition, by relying on both the transition kernel and the measurement kernel, provides a way to determine filter stability when the transition kernel alone does not satisfy the mixing condition or variants of it.
\end{remark}
\begin{remark}
Our result result is sufficient, but certainly not necessary.
In what seems like a counter-intuitive result, this result prioritizes measurement channels $Q$ that are un-informative as opposed to those that are informative (see \cite{mcdonald2018stability} for more discussion on informative measurement channels). For example a completely independent observation $Y$ will have $\delta(Q)=1$ and direct observation will have $\delta(Q)=0$. However, the idea of our result is that the mapping $T$ is a contraction without any Bayes update. We then want a measurement kernel $Q$ that does not ``change'' this ergodic property, and a completely independent observation will result in $\psi(\mu)=\mu$, and hence will not conflict with the transition kernel $T$.

For example, consider a finite system and direct observation. That is $y$ is an invertible deterministic function of $x$, $Y=h(X)$. Then we have
\begin{align*}
\|\psi(\mu,y)-\psi(\nu,y)\|
=&\sup_{\|f\|_{\infty}\leq 1}\left|\sum_{x \in\mathcal{X}}f(x)g(x,y)\left(\frac{\mu(x)}{N^{\mu}(x)}-\frac{\nu(x)}{N^{\nu}(y)}\right)\right|\\
=&\sup_{\|f\|_{\infty}\leq 1}\left|\sum_{x\in \mathcal{X}}f(x)1_{h^{-1}(y)}(x)\left(\frac{\mu(x)}{\mu(h^{-1}(y))}-\frac{\nu(x)}{\nu(h^{-1}(y))}\right)\right|\\
=&\sup_{\|f\|_{\infty}\leq 1}\left|f(h^{-1}(y))\left(\frac{\mu(h^{-1}(y))}{\mu(h^{-1}(y))}-\frac{\nu(h^{-1}(y))}{\nu(h^{-1}(y))} \right) \right|
=0
\end{align*}

However, if we add and subtract $\frac{\mu(h^{-1}(y))}{\nu(h^{-1}(y))}$ in the first line and apply the triangle inequality we instead have:
\begin{small}
\begin{align*}
&\left(\frac{1}{\mu(h^{-1}(y))} \right)\sup_{\|f\|_{\infty}\leq 1}\left|f(h^{-1}(y))(\mu(h^{-1}(y))-\nu(h^{-1}(y)) \right|
+|f(h^{-1}(y))\nu(h^{-1}(y))|\left|\frac{\mu(h^{-1}(y))-\nu(h^{-1}(y))}{\nu(h^{-1}(y))} \right|\\
=&\left(\frac{1}{\mu(h^{-1}(y))} \right)\left|(\mu(h^{-1}(y))-\nu(h^{-1}(y)) \right|+\left|\mu(h^{-1}(y))-\nu(h^{-1}(y))\right|
\neq 0
\end{align*}
\end{small}
this is the same approach taken in the proof of Lemma \ref{Bayes_Dobrushin}. We see that the triangle inequality results in a loose bound that ignores the informative nature of the measurement channel, and thus Theorem \ref{mixed_Dobrushin} relies on the ergodic properties of the transition kernel to achieve exponential filter stability and the measurement kernel to not interfere.
\end{remark}
\begin{remark}
In some cases we are interested in almost sure statements about the path wise convergence of the filter. Similar to \cite[Theorem 2, Part 2]{kleptsyna2008discrete}, filter stability in a pathwise sense follows from exponential stability in expectation via Markov inequality and Borel Cantelli Lemma. Thus if the filter process is exponentially stable with coefficient $\alpha <1$, for any $\rho <\frac{1}{\alpha}$ we also have $
\rho^{k}\|\pi_{n}^{\mu}-\pi_{n}^{\nu}\|_{TV} \to 0~P^{\mu}~a.s.
$
\end{remark}

\section{Controlled Case}\label{controlledCase}
Many applications of filtering involve controlled dynamics, where very few results on filter stability have been reported. In the controlled environment considered, the measurement channel $Q$ is unchanged, however the transition kernel $T(dx'|x,u)$ is different for each applied control action $u$. 

In a controlled process, we must modify the definition of the filter to be conditioned on both past measurements and control actions, that is $\pi_{n}^{\mu}(\cdot)=P^{\mu}(X_{n} \in \cdot|Y_{[0,n]},U_{[0,n-1]})$. With knowledge of the past control actions taken, the filter is still recursive in the following fashion:
\begin{align*}
\pi_{n+1}^{\mu}(dx)&=\phi(\pi_{n}^{\mu},u_{n},y_{n+1})=\frac{g(x,y_{n+1})\int_{\mathcal{X}}T(dx|x',u_{n})\pi_{n}^{\mu}(dx')}{\int_{\mathcal{X}}g(x,y_{n+1})\int_{\mathcal{X}}T(dx|x',u_{n})\pi_{n}^{\mu}(dx')}
\end{align*}
in the update the only difference is $T(dx|x',u_{n})$ now depends on the past control action.

If we define
$
\tilde{\delta}(T)=\inf_{u \in \mathcal{U}} \delta(T(\cdot|\cdot,u))
$
then the result for a controlled model follows immediately from the proof of Theorem \ref{mixed_Dobrushin}.

\begin{theorem}
Assume $\mu \ll \nu$ and that the measurement channel $Q$ is dominated. If we have
$
\alpha=(1-\tilde{\delta}(T))(2-\delta(Q))<1
$
then the filter is exponentially stable with coefficient $\alpha$ for any control policy.
\end{theorem}

Therefore, in order to guarantee  exponential stability in a control environment we first check the expansion coefficient of the Bayesian update operator $(2-\delta(Q))$. Then, we find the Dobrushin coefficient of $T(\cdot|\cdot,u)$ for every different control action $u$. If under each control action $T(\cdot|\cdot,u)$ has a high enough Dobrushin coefficient, then for every control action the filter update operator is a contraction in total variation in expectation.

It is important to emphasize that it then does not matter what control policy is implemented, since each control action results in a transition kernel with a sufficiently high Dobrushin coefficient, and thus we have uniform exponential stability over all control policies. 

\section{An Application}\label{CaseStudySection}
Consider a system where $\mathcal{X}=\mathcal{Y}=\mathbb{R}$ and the transition and measurement kernels are defined by the functions
\begin{align*}
x_{n+1}=f(x_{n})+N(0,\sigma_{t}^{2})\quad \quad y_{n}= g(x_{n})+N(0,\sigma_{q}^{2})
\end{align*}
that is an additive Gaussian system, but not necessarily a linear one. Assume the functions $f$ and $g$ are measurable and bounded with norms $f(x)\in [-t,t]$ and $g(x) \in [-q,q]$. We then have that
\begin{align*}
&T(dx_{n+1}|x_{n})\sim N(f(x_{n}),\sigma_{t}^{2})&Q(dy_{n}|x_{n})\sim N(f(x_{n}),\sigma_{q}^{2})
\end{align*}
This is not a mixing system in the sense of the conditions required to be able to invoke Hilbert metric based methods (see Definition \ref{mixingDefinition}), hence the previous results in the literature would not apply. Furthermore, $f$ and $g$ are not necessarily well behaved Lipschitz and invertible functions, hence the results of \cite{crisan2008stability} do not apply either. For these kernels we have that
\begin{align*}
&\delta(T)=2P(N(t,\sigma_{t}^{2})<0)&\delta(T)=2P(N(q,\sigma_{q}^{2})<0)
\end{align*}
and this probability is fully determined by the ratio of the mean and standard deviation of the Gaussian in question, $\frac{\sigma_{t}}{t}$, and $\frac{\sigma_{q}}{q}$. The higher the ratio, the higher the Dobrushin coefficient. In Table \ref{tab:values} we see a list of the ratio of the transition kernel and lowest possible ratio of the measurement kernel such that 
$
(1-\delta(T))(2-\delta(Q))<1
$.
If the ratio of $\frac{\sigma_{q}}{q}$ is higher than the stated value, we will get exponential stability for the given transition kernel. If $\frac{\sigma_{t}}{t}> 1.5$ then $\delta(T)> \frac{1}{2}$ and we have exponential stability regardless of $Q$.
\begin{table}[h]
\begin{footnotesize}
\begin{tabular}{|c|c|c|c|c|c|c|c|c|c|c|c|c|c|c|}
\hline
$\frac{\sigma_{t}}{t}$&1.5&1.4&1.3&1.2&1.1&1.0&0.9&0.8&0.7&0.6&0.5&0.4&0.3\\ \hline
$\frac{\sigma_{q}}{q}$&N/A&0.6&0.8&1.01&1.3&1.65&2.13&3.25&5.5&8.0&20.0&70.0
&1000.0\\ \hline
$\delta(T)$&0.50&0.48&0.44&0.40&0.36&0.32&0.27&0.21&0.15&0.10&0.05&0.01 &0.00\\ \hline
$\delta(Q)$&N/A&0.10&0.21&0.32&0.44&0.54&0.64&0.76&0.86&0.90&0.96&0.99 &1.00\\ \hline
\end{tabular}
\caption{Approximate minimum ratio of $\frac{\sigma_{q}}{q}$ in order to achieve a contraction for low values of the transition kernel ratio.}
\label{tab:values}
\end{footnotesize}
\end{table}

\section{Conclusion}\label{sec:conclusion}
In this paper, we propose an alternative approach for exponential stability, where our approach builds on utilizing the Dobrushin's ergodic coefficients associated with both the transition kernel as well as the measurement channel. Such a joint study seems to have been unexplored in the literature, and leads to a concise analysis and simple explicit conditions on filter stability which can be applied to more general system models, including controlled stochastic models. 

\bibliographystyle{amsplain}
\bibliography{thesis}

\end{document}